\documentclass[11pt,a4paper]{amsart}
\usepackage{amsmath,amsfonts,amssymb,amsthm}
\usepackage{mathrsfs}
\usepackage[all,cmtip]{xy}
\usepackage{xcolor}
\usepackage{marginnote}
\usepackage[normalem]{ulem}

\usepackage[left=1.0in,right=1.0in,top=1.0in,bottom=1.0in]{geometry}

\newtheorem{theorem}{Theorem}[section]
\newtheorem{lemma}[theorem]{Lemma}
\newtheorem{corollary}[theorem]{Corollary}
\newtheorem{proposition}[theorem]{Proposition}

\newtheorem{fact}[theorem]{Fact}

\theoremstyle{definition}

\newtheorem{definition}[theorem]{Definition}
\newtheorem{remark}[theorem]{Remark}

\newtheorem*{claim*}{Claim}

\newcommand{\Z}{{\mathbb Z}}
\newcommand{\N}{{\mathbb N}}
\newcommand{\R}{{\mathbb R}}

\def\kal#1{\mathcal{K}_{#1}}   
\def\lkal#1{\ell\kal{#1}} 

\def\hull#1{\langle{#1}\rangle} 

\def\span#1{\operatorname{span}({#1})}
\def\linespan#1{\operatorname{span}(#1)}

\title{Topological groups with invariant linear spans}

\author[E. Perneck\'a]{Eva Perneck\'a}
\address[Eva Perneck\'a]{Department of Applied Mathematics\\ Faculty of Information Technology\\ Czech Technical University in Prague\\ Th\'akurova 9, 16000, Prague 6\\ Czech Republic} \email{perneeva@fit.cvut.cz}

\author[J. Sp\v{e}v\'ak]{Jan Sp\v{e}v\'ak}
\address[Jan Sp\v{e}v\'ak]{Department of Applied Mathematics\\ Faculty of Information Technology\\ Czech Technical University in Prague\\ Th\'akurova 9, 16000, Prague 6\\ Czech Republic} \email{spevajan@fit.cvut.cz}

\thanks{The first author was supported by the grant GA\v CR 18-00960Y of the Czech Science Foundation.}

\begin{document}

\begin{abstract}
Given a topological group $G$ that can be embedded as a topological subgroup into some topological vector space (over the field of reals) we say that $G$ has { \it invariant linear span} if all linear spans of $G$ under arbitrary embeddings into topological vector spaces are isomorphic as topological vector spaces.

For an arbitrary set $A$ let $\Z^{(A)}$ be the direct sum of $|A|$-many copies of the discrete group of integers endowed with the Tychonoff product topology.
We show that the topological group  $\Z^{(A)}$ has invariant linear span. This answers a question from {\cite{DSS}} in positive. 

We prove that given a non-discrete sequential space $X$, the free abelian topological group $A(X)$  over $X$  is an example of a topological group that embeds into a topological vector space but does not have invariant linear span.

\end{abstract}

\maketitle

\def\rank#1{trank(#1)}

\bigskip

{All vector spaces in this paper are considered over the field $\R$ of real numbers and all topological spaces are assumed to be Hausdorff. For an arbitrary non-empty set $A$ and a topological group $G$ with addition and neutral element $0_G$ let $G^A$ be the topological group given by the direct product $\Pi_{a\in A}G$ with coordinate-wise addition and the Tychonoff product topology. We denote $G^{(A)}$ the topological subgroup of $G^A$, with inherited topology, consisting of those elements $(g_a)_{a\in A}$ for which the set $\{a\in A\colon g_a\neq 0_G\}$ is finite. Given a subset $H$ of a group $G$ and a subset $M$ of a vector space $L$, we use the standard notation $\hull{H}$ to denote the subgroup of $G$ generated by $H$ and $\span{M}$ for the vector subspace of $L$ generated by $M$. For simplicity we write $\hull{g}$ rather than $\hull{\{g\}}$ for any $g\in G$ and, similarly, $\linespan{l}$ instead of $\span{\{l\}}$ for any $l\in L$.}

\section{Introduction}
{
In this note we study which topological groups enjoy the property stated in the following definition.
}

\begin{definition}
Let $G$ be a topological group that can be embedded (as a topological subgroup) into some topological vector space. We say that $G$ has  {\it invariant linear span} provided that all linear spans of $G$ under arbitrary {embeddings into topological vector spaces} are 
isomorphic as topological vector spaces.
\end{definition}
A simple example of topological group with an invariant linear span is every topological vector space. Indeed, as was observed by Tkachuk in \cite{Tkachuk}, given arbitrary topological vector spaces $L$ and $E$ and  a continuous group homomorphism  $h:L\to E$, the homomorphism $h$ is automatically linear. This observation further yields that if $L$ is embedded in $E$ as a topological subgroup, the same embedding is already an embedding of topological vector spaces. In particular, the linear span of $L$ in $E$ is (isomorphic to) the topological vector space $L$ again and hence the linear span of $L$ does not depend on the space $E$ in which $L$ embeds. Yet another simple example of a topological group with an invariant linear span is the discrete topological group $\Z$ of integers. Its linear span is obviously (isomorphic to) the topological vector space $\R$.

In our paper we show that for an arbitrary non-empty set $A$ the {group $\Z^{(A)}$} has invariant linear span (which is isomorphic to $\R^{(A)}$). See Theorem \ref{lkal:kal:both:embed} and Corollary \ref{cor:main}. This answers  \cite[Question 10.6]{DSS} in positive and generalizes the folklore fact that all topological vector spaces of the same finite dimension are isomorphic (see Remark \ref{remark}).

The proof of Theorem \ref{lkal:kal:both:embed} consists of two steps. The first was done in \cite[Proposition 10.1]{DSS} by showing that given an injective linear map $l:\R^{(A)}\to L$, where $L$ is a topological vector space, the continuity of $l$ follows from the continuity of the restriction of $l$ to $\Z^{(A)}$. The second step is done in Theorem \ref{thm}, where we basically show, that if the restriction of $l$ to $\Z^{(A)}$ is an embedding of topological groups, then $l$ is open. The proof of Theorem \ref{thm} is based on a Diophantine approximation done in Lemma \ref{lemma:Kronecker} which resembles the classical Kronecker's approximation theorem.

We end the paper with Theorem \ref{proposition:tkachenko:Morris:Gabriyelyan}, which shows that for an arbitrary non-discrete sequential space $X$ the free topological abelian group $A(X)$ does not have invariant linear span, as it canonically embeds in both the free topological vector space $V(X)$ and  the free locally convex topological vector space $L(X)$, and the linear spans of $A(X)$ in the latter spaces are the non-isomorphic topological vector spaces $V(X)$ and $L(X)$. This theorem is based on non-trivial results of Tkachenko \cite{Tka} and Gabriyelyan and Morris \cite{GM}.

\section{The main technical theorem}

{We begin the section by two auxiliary observations.}

\begin{lemma}\label{Kronecker1}
For every neighbourhood $V$ of zero in a compact group $G$  and every $t\in G$ there is $m\in\N\setminus\{0\}$ such that 
\begin{equation}\label{Lt:in:V}
mt\in V.    
\end{equation}

\end{lemma}
\begin{proof}
Pick a neighbourhood $V$ of zero in $G$ and $t\in G$ arbitrarily. There are two possibilities. If  $\hull{t}$ is a discrete subgroup of $G$ then it is closed and, consequently, compact and therefore finite. Let $m$ be the order of  $\hull{t}$ and observe that \eqref{Lt:in:V} holds. The second possibility is, that  $\hull{t}$ is not discrete. Then every neighbourhood of zero (and $V$ in particular) contains infinitely many elements of $\hull{t}$. Since $t$ is a generator of $\hull{t}$ there is $m\in\N\setminus\{0\}$ satisfying \eqref{Lt:in:V}.
\end{proof}

\begin{lemma}\label{lemma:Kronecker}
Let $(t_n)_{n\in\N}$ be a sequence in $\R^F$, where $F$ is a finite set. For every neighborhood $O$ of zero in $\R^F$ there is $m\in\N\setminus\{0\}$ and a sequence $(z_n)_{n\in\N}\subset\Z^F$ such that 
\begin{equation}\label{eq:kronecker}
    mt_n-z_n\in O
\end{equation}
holds for infinitely many $n\in\N$.

\end{lemma}
\begin{proof}
Fix $O$, a neighbourhood of zero in $\R^F$, arbitrarily, and let $q:\R^F\to(\R/\Z)^F$ be the quotient map. Since $(\R/\Z)^F$ is sequentially compact, the sequence $(q(t_n))_{n\in\N}$ has a convergent subsequence with a limit $t$. As $q$ is an open map, we may pick a neighbourhood $V$ of zero in $(\R/\Z)^F$ such that $$V+V\subset q(O).$$ By Lemma \ref{Kronecker1}, there is a positive integer $m$ satisfying \eqref{Lt:in:V}.  Observe that the set $M$ defined as $$M:=\{n\in\N:mq(t_n)\in mt+V\}$$ is infinite, and for every $n\in M$ we have $$q(mt_n)=mq(t_n)\in mt+V\subset V+V\subset q(O). $$ Thus for every $n\in M$ there is $z_n\in\Z^F$ such that \eqref{eq:kronecker} holds. 
\end{proof}

In order to {formulate} the main technical result of this paper we need to recall three notions. Their importance to the topic of our manuscript will become clear from Proposition \ref{prop:abs:cauch:+:semibasic=embedding} and from the proof of Theorem \ref{lkal:kal:both:embed}.

We say that a subset $A$ of {a} topological vector space $L$ is 
 \begin{itemize}
     \item
          {\em absolutely Cauchy summable\/} provided that for every neighbourhood $V$ of $0_L$ 
          there exists a finite set $F{\subset} A$ such that    
             \begin{equation}\label{eq5}
             {\span{A\setminus F}}\subset V;
              \end{equation}
     \item
          {\em topologically independent} if $0_L\not\in A$ and  for every neighbourhood $W$ of $0_L$ 
          there exists {a} neighbourhood $U$ of {$0_L$} such that for every finite subset $F\subset A$ and every indexed set $\{z_a:a\in F\}$ of integers the inclusion $\sum_{a\in F}z_aa\in U$ implies that $z_aa\in W$ for all $a\in F$. We call this neighbourhood $U$ a {\em $W$-witness\/} of the topological independence of $A$; 
     \item
           {\em semi-basic} if for all $a\in A$ we have
 \begin{equation}\label{eq:sembasic}
 a\not\in{\overline{\span{A\setminus\{a\}}}}.
 \end{equation}

 \end{itemize}

\begin{remark}
 In \cite[Definition 3.1]{DSS} the notion of an absolutely Cauchy summable set was introduced in an arbitrary abelian topological group. In topological vector spaces it is equivalent to our definition by \cite[Proposition 9.2 (i)]{DSS}.
 
 {\em Topologically independent sets} were introduced in \cite[Definition 4.1]{DSS} in an arbitrary abelian topological group. For further properties of these sets in precompact groups we refer to \cite{Spe}. 
  
  We have {adopted} the name {\em semi-basic} from \cite{kal2}, where a semi-basic sequence in an $F$-space was introduced. In \cite{Haz} a semi-basic set is called {\em topologically free}.  Semi-basic sequences in Banach spaces are called {\em minimal} in \cite{H} and \cite[Definition 6.1]{Singer}.
\end{remark}

Now we are ready to state the main technical theorem of this note.
\begin{theorem}
\label{thm}
If $A$ is a topologically independent and absolutely Cauchy summable subset of a topological vector space $L$, then $A$ is semi-basic.
\end{theorem}
\begin{proof}
To prove the contrapositive, assume that there is $a\in A$ with
 \begin{equation}\label{eq:a:in:rhull}
a\in{\overline{\span{A\setminus\{a\}}}},   
 \end{equation}
 and let $A$ be absolutely Cauchy summable. We will show that $A$ is not topologically independent.
 
If $a=0_L$, then we are done. Otherwise 
 we can find a neighbourhood $W$ of $0_L$ such that $za\notin W$ for every $z\in\Z\setminus\{0\}$. Pick an arbitrary neighborhood $U$ of $0_L$.
Let us show that $U$ is not a $W$-witness of topological independence of $A$.  

Fix a balanced neighborhood $V$ of $0_L$ with $V+V+V\subset U$. Since $A$ is absolutely Cauchy summable, there is a finite $F\subset A\setminus\{a\}$ such that ${\span{A\setminus (F\cup\{a\})}}\subset V.$ In particular, for every finite $B\subset A\setminus\{a\}$, reals $(s_b)_{b\in B\setminus F}$ and each $n\in\N$ we have 
\begin{equation}\label{eq:in:nth:v1}
\sum_{b\in B\setminus F} s_bb\in\frac{1}{n}V.
\end{equation}

Given $n\in\N$ arbitrarily, by \eqref{eq:a:in:rhull} we can fix a finite set $B\subset A\setminus\{a\}$ and an indexed set $(r_b^n)_{b\in B}$ of reals such that 
\begin{equation}\label{eq:in:nth:v2}
a-\sum_{b\in B} r^n_b b\in \frac{1}{n}V.    
\end{equation}

For $b\in F\setminus B$ define $r^n_b=0$, and observe that \eqref{eq:in:nth:v1} and \eqref{eq:in:nth:v2} yield
\begin{equation}
    \label{eq:n_approx}
  a-\sum_{b\in F} r^n_b b=\left(a-\sum_{b\in B} r^n_b b\right)+\left(\sum_{b\in B\setminus F} r^n_b b\right)      \in \frac{1}{n}V+\frac{1}{n}V.  
\end{equation}

By continuity of vector space operations, there is a neighborhood  $O$ of zero in $\R^F$ such that 
\begin{equation}\label{eq:99}
    \sum_{b\in F}s_bb\in V \mbox{ for all } (s_b)_{b\in F}\in O.
\end{equation}
Define a sequence $(t_n)_{n\in\N}$ in $\R^F$ by $t_n=(r^n_b)_{b\in F}$, and let $m\in\N\setminus\{0\}$ and 
$(z_n)_{n\in\N}\subset\Z^F$
be as in the conclusion of Lemma \ref{lemma:Kronecker}. By this lemma, we may fix $n\in\N$ such that $n\geq m$ and \eqref{eq:kronecker} holds. For $b\in F$ let $z_b\in\Z$ be the $b$-th coordinate of $z_n$, and observe that by \eqref{eq:kronecker} and \eqref{eq:99} we have 
$$\sum_{b\in F} (mr^n_b-z_b) b\in V.$$
From this, \eqref{eq:n_approx}, and the fact that $V$ is balanced and $n\ge m$ we get

$$ ma-\sum_{b\in F} z_b b=m\left(a-\sum_{b\in F} r^n_b b\right)+\left(\sum_{b\in F} (mr^n_b-z_b) b\right)      \in \frac{m}{n}V+\frac{m}{n}V+V\subset V+V+V\subset U.   $$ Since $m$ is a non-zero integer and $z_b$ is an integer for each $b\in F$ we conclude that $U$ is not a $W$-witness of the topological independence of $A$, because $ma\not\in W$ by the choice of $W$.
\end{proof}

\section{The invariance of {the} linear span of $\Z^{(A)}$ }

In this section we prove that  the topological group $\Z^{(A)}$ has invariant linear span. In order to do so we need to recall the notion of a (linear) Kalton map introduced in \cite{DSS} which is useful to deal with embedding{s} of $\Z^{(A)}$ and $\R^{(A)}$ into topological vector spaces. 

{
Given a non-empty subset $A$ of a topological vector space $L$ such that $0\not\in A$, we denote
$$\kal{A}:\Z^{(A)}\to L $$
the group homomorphism given by $\kal{A}\left((z_a)_{a\in A}\right)=\sum_{a\in A}z_a a$ for every $(z_a)_{a\in A}\in \Z^{(A)}$.
Similarly,  
$$\lkal{A}:\R^{(A)}\to L$$
is the linear operator between vector spaces defined by $\lkal{A}\left((r_a)_{a\in A}\right)=\sum_{a\in A}r_a a$ for every $(r_a)_{a\in A}\in \R^{(A)}$. As in \cite{DSS} we call $\kal{A}$ ($\lkal{A}$) the (linear) Kalton map associated with $A$. Since the sums in the definitions are finite, the mappings are well-defined and $\kal{A}(\Z^{(A)})=\hull{A}\subset L$ and $\lkal{A}(\R^{(A)})=\span{A}\subset L$. Notice that the (linear) Kalton map is injective if and only if $A$ is (linearly) independent.
}

\begin{fact}[{\cite[Proposition 10.1]{DSS}}]\label{fact:cauchy:summable}
Given a {non-empty} subset $A$ of non-zero elements of a topological vector space, the following statements are equivalent:
\begin{itemize}
    \item [(i)]
    the linear Kalton map $\lkal{A}$ is continuous;
    \item[(ii)]
    the Kalton map $\kal{A}$ is continuous;
    \item[(iii)]
    the set $A$ is absolutely Cauchy summable.
    \end{itemize}
\end{fact}

\begin{lemma}\label{fact:}\label{facta:}
Let $A$ be a {non-empty} subset of a topological vector space. The following conditions are equivalent:
\begin{itemize}
    \item[(i)] the linear Kalton map $\lkal{A}$ is an open injection onto ${\span{A}}$; 
    \item[(ii)]
     the set $A$ is semi-basic.
\end{itemize}

\end{lemma}
\begin{proof}
Observe that from both items (i) and (ii) it follows that $A$ is linearly independent. Therefore, if we assume either (i) or (ii), then for each $a\in A$ there is a unique {linear} projection   $\pi^A_a:{\span{A}\to\linespan{a}}$ such that $\ker\pi^A_a={\span{A\setminus\{a\}}}$ and $\pi^A_a$ restricted to ${\linespan{a}}$ is the identity map.

To end the proof it suffices to show that items (i) and (ii) are  both equivalent to the following fact for a linearly independent set $A$:
\begin{equation}\label{eq:proj:A:a}
\mbox{ $\pi^A_a:{\span{A}\to\linespan{a}}$ is continuous for every $a\in A$. }    
\end{equation}

The equivalence of (i)  and \eqref{eq:proj:A:a} follows from \cite[Proposition 10.2]{DSS}. To establish the equivalence of (ii) and \eqref{eq:proj:A:a} it suffices to realize that the  continuity of each $\pi^A_a$ is  equivalent to the fact that each $\ker(\pi^A_a)$ is closed in ${\span{A}}$ and this happens if and only if \eqref{eq:sembasic} holds for all $a\in A$.
\end{proof}

\begin{proposition}\label{prop:abs:cauch:+:semibasic=embedding}
Given a subset $A$ of a topological vector space, the following conditions are equivalent:
\begin{itemize}
    \item[(i)] 
    The linear Kalton map $\lkal{A}$ is an embedding of topological vector spaces.
    \item[(ii)]
    $A$ is absolutely Cauchy summable and semi-basic. 
\end{itemize}
\end{proposition}
\begin{proof}
Assume (i). Then $A$ is absolutely Cauchy summable by Fact \ref{fact:cauchy:summable} and semi-basic by Lemma \ref{fact:}. Thus (ii) holds. 

{If (ii) holds, Lemma \ref{fact:} implies, that the linear Kalton map $\lkal{A}$ is open and injective, while Fact \ref{fact:cauchy:summable} provides its continuity. This gives us (i).}
\end{proof}

Our next theorem answers \cite[Question 10.6]{DSS} in positive.
\begin{theorem}\label{lkal:kal:both:embed}
Given a subset $A$ of a topological vector space the following statements are equivalent:
\begin{itemize}
\item[(i)]
The Kalton map $\kal{A}$ is an embedding of topological groups;
\item[(ii)]
The linear Kalton map $\lkal{A}$ is an embedding of topological vector spaces.
\end{itemize}
\end{theorem}
\begin{proof}
Since $\kal{A}$ is a restriction of $\lkal{A}$, the implication (ii)$\Rightarrow$(i) follows. 

Assume (i). Then $A$ is absolutely Cauchy summable by Fact \ref{fact:cauchy:summable}. Further, $A$ is topologically independent by \cite[Proposition 4.7 (ii)]{DSS}. Theorem \ref{thm} yields that $A$ is also semi-basic. To show (ii) it remains to apply Proposition~\ref{prop:abs:cauch:+:semibasic=embedding}.
\end{proof}
The next statement is a direct corollary of Theorem \ref{lkal:kal:both:embed}.

 \begin{corollary}\label{cor:main}
For every non-empty set $A$ the topological group $\Z^{(A)}$ has invariant linear span (which is isomorphic to  $\R^{(A)}$).
  \end{corollary}
  
\begin{remark}\label{remark}
Corollary \ref{cor:main} can be viewed as a generalization of the folklore fact that {\em all topological vector spaces of the same finite dimension are isomorphic}. Indeed, if $A$ is a finite basis of a topological vector space $V$, then $A$ is topologically independent by \cite[Proposition 4.11]{DSS}. {It follows then by \cite[Proposition 4.8]{DSS} that the Kalton map $\kal{A}$ is an embedding of topological groups}. That is, the hull $\hull{A}$ is (isomorphic to) $\Z^A$. Hence $V=\span{A}$ is (isomorphic to) $\R^A$.

\end{remark}

We end this paper with a theorem which provides a rich source of examples of topological groups that embed in topological vector spaces and do not have invariant linear spans.

Given a Tychonoff space $X$ the symbols $A(X)$, $L(X)$ and $V(X)$ stand for the free abelian topological group, the free locally convex topological vector space and the free topological vector space over $X$ respectively.  We refer {the} reader to \cite{GM} for definitions of these notions. 

\begin{theorem}\label{proposition:tkachenko:Morris:Gabriyelyan}
Let $X$ be a Tychonoff space. The topological group $A(X)$ canonically embeds in the topological vector spaces $L(X)$ and $V(X)$. If $X$ is sequential and non-discrete, then $A(X)$ does not have an invariant linear span. 
\end{theorem}
\begin{proof}
By \cite[Theorem 3]{Tka} the topological group $A(X)$ embeds in $L(X)$ and the linear span of $A(X)$ is $L(X)$. On the other hand, by \cite[Proposition 5.1]{GM}, it also embeds in $V(X)$ and its linear span in $V(X)$ is $V(X)$. Finally, if $L(X)$ and $V(X)$ are isomorphic as topological vector spaces and $X$ is sequential, then $X$ is discrete by \cite[Corollary 4.5]{GM}.   
\end{proof}

\end{document}